\documentclass[11pt,a4paper]{amsart}
\usepackage{a4wide}
\usepackage{amsfonts, amsthm, amsmath, amssymb, enumerate}
\usepackage{algorithm, algpseudocode, algorithmicx}
\usepackage{booktabs, array, colortbl, multirow}
\usepackage{graphicx}
\usepackage{tikz, pgfplots}\pgfplotsset{compat=newest}
\usepackage{url}
\DeclareRobustCommand{\pmod}[1]{\allowbreak\mathchoice{\mskip6mu}{\mskip3mu}{\mskip1mu}{\mskip0mu}\text{(mod$\mskip6mu#1$)}}
\def\sfname#1{\text{\sffamily\upshape#1}}
\DeclareRobustCommand{\gcd}{\mathchoice{\sfname{\small GCD}}{\sfname{\fontsize{8}{4}\selectfont GCD}}{\sfname{\tiny GCD}}{\sfname{\fontsize{5}{2}\selectfont GCD}}}
\newtheorem*{LehmerTheorem}{Lehmer Theorem}
\newtheorem{MainTheorem}{Theorem}

\newtheorem{theorem}{Theorem}[section]
\newtheorem{lemma}[theorem]{Lemma}
\newtheorem{corollary}[theorem]{Corollary}
\newtheorem{prop}[theorem]{Proposition}
\newtheorem{conj}[theorem]{Conjecture}
\theoremstyle{definition}
\newtheorem{definition}[theorem]{Definition}
\newtheorem{remark}[theorem]{Remark}

\newcommand{\N}{\ensuremath{\mathbb{N}}}
\newcommand{\R}{\ensuremath{\mathbb{R}}}
\newcommand{\Z}{\ensuremath{\mathbb{Z}}}
\newcommand{\ZZ}{\ensuremath{\mathcal{Z}}}
\newcommand{\SI}{\ensuremath{\mathbb{S}^1}}
\newcommand{\floor}[1]{\ensuremath{\lfloor #1\rfloor}}
\newcommand{\elasticfloor}[1]{\ensuremath{\left\lfloor #1 \right\rfloor}}

\newcommand{\elasticset}[2]{\ensuremath{\left\{#1 \,\colon #2\right\}}}
\makeatletter
\def\@map#1#2[#3]{\mbox{$#1 \colon #2 \longrightarrow #3$}}
\def\map#1#2{\@ifnextchar [{\@map{#1}{#2}}{\@map{#1}{#2}[#2]}}
\makeatother
\title{Mersenne numbers and the doubling map}
\date{April 21, 2026}

\author{Llu\'{\i}s Alsed\`a, Antonio Garijo \and Xavier Jarque}
\address{Llu\'{\i}s Alsed\`a:
         Departament de Matem\`atiques,
         Edifici Ciències,
         Universitat Autònoma de Bar\-ce\-lo\-na,
         08193 Cerdanyola del Vallès, Barcelona, Spain}
\email{lluis.alseda@uab.cat}

\address{Antonio Garijo:
Departament d'Enginyeria Inform\`atica i Matem\`atiques,
Universitat Rovira i Virgili, 43007 Tarragona, Spain}
\email{antonio.garijo@urv.cat}

\address{Xavier Jarque:
Departament de Matem\`atiques i Inform\`atica,
Universitat de Barcelona,
Gran Via, 585,
08007 Barcelona and
Centre de Recerca Matem\`atica,
Edifici Ciències, Campus de Bellaterra,
Universitat Autònoma de Barcelona,
08193 Bellaterra, Catalonia}
\email{xavier.jarque@ub.edu}

\thanks{\emph{Acknowledgements:\/}
The first author is supported by the grant PID2023-146424NB-I00 of Ministerio de Ciencia e Innovaci\'on.
The second author is supported by the project 2021SGR-633 funded by Generalitat de Catalunya.
The third author is supported by the project PID2023-147252NB-I00 funded by
MICIU/AEI MCIN/AEI/10.13039/501100011033 and by FEDER, UE and it is also funded by the
Spanish State Research Agency, through the Severo Ochoa and Mar\'ia de Maeztu Program for Centers and Units of Excellence in R\&D (CEX2020-001084-M)}
\subjclass{Primary: 37E10, 11Y55 }
\keywords{Mersenne numbers, number theory, doubling map, periodic points}
\begin{document}
\begin{abstract}
We study the connection between the \emph{Mersenne numbers} $M(n) = 2^n - 1$ and
the dynamics of the angle-doubling map $\theta \mapsto 2\theta \pmod{1}$.
Within this framework, we develop an algorithm to compute divisors of
Mersenne numbers without explicitly evaluating $M(n)$.
Determining whether $M(n)$ is prime for a prime $n$
(and knowing if there are infinitely many of them), is a central problem,
traditionally addressed with the help of the Lucas–Lehmer test.
We provide an alternative approach based on dynamical methods.
As an application, we prove that $M(2{,}199{,}023{,}254{,}451)$
(with approximately $6.6 \times 10^{11}$ digits) is composite by exhibiting a
non-trivial divisor.
\end{abstract}
\maketitle
\section{Introduction}\label{sec:introduction}
A central goal in mathematics is the study of sequences of natural numbers
defined by a variety of properties, such as arithmetic or primality conditions,
congruences, or recursive formulas.
The jewel in the crown is the sequence of prime numbers itself,
which is deeply connected to one of the most famous conjectures in mathematics,
\emph{the Riemann Hypothesis}. Indeed, the Fundamental Theorem of Arithmetic states
that every natural number can be expressed uniquely as a product of prime numbers.
Moreover, large prime numbers play a crucial role in modern cryptography,
see \cite{RSA,DH,Mughal2024}.

The French mathematician Marin Mersenne (1588 -- 1648) \cite{Mer,Mer2} was the pioneer
in studying the infinite sequence of natural numbers $M(n) := 2^n-1$,
today known as \emph{Mersenne numbers}.

One can easily check that if $a,b,c \in \N$, then
\[
  a^{bc} - 1 =
    	\Bigl(a^b - 1\Bigr)\Bigl(1+a^b+a^{2b}+\dots+a^{(c-1)b}\Bigr) =
	    \Bigl(a^c - 1\Bigr)\Bigl(1+a^c+a^{2c}+\dots+a^{(b-1)c}\Bigr).
\]
Consequently, if  $n=bc$ is composite, then $M(n)$ is also composite. However, when
$n$ is prime, the corresponding Mersenne number $M(n)$ may or may not be prime. For example,
$M(2)=3$, $M(3)=7$ and $M(31)=2,147,483,647$ are prime,
whereas $M(11)=2^{11}-1=2,047=23 \times 89$ is composite.
It still remains as an open question whether the sequence of Mersenne primes is finite or infinite.

The largest prime number known to date is the Mersenne number $M(136,279,841)$,
discovered by Luke Durant in October 2024.
As a curiosity, this number has approximately 41 million decimal digits.
Currently, the exponent $n=77,232,917$ (corresponding to the 50th Mersenne prime)
is the largest exponent below which all candidates to be prime Mersenne numbers have been independently verified.
It has not yet been confirmed whether any undiscovered Mersenne primes exist between $M(77,232,917)$ (the 50st) and
$M(136,279,841)$ (the 52nd).
For the most up-to-date information, see the Great Internet Mersenne Prime Search \cite{GIMPS}.

The GIMPS project has discovered a total of eighteen Mersenne primes,
sixteen of which were the largest known prime numbers at the time of their discovery.
Clearly, the search for large prime numbers is closely tied to the search for large Mersenne primes.

Although there are many theoretical equivalences to the primality of $M(n)$
there are no simple tests (in the sense of being useful for computational purposes)
to determine whether $M(n)$ is prime for a given prime exponent $n$.

The GIMPS search is based on the Lucas-Lehmer test \cite{lehmer1935lucas, BLS}, a direct consequence of the following theorem.

\begin{LehmerTheorem}
Let $p>2$ be a prime number.
The Mersenne number $M(p)=2^p-1$ is prime if and only if $S_{p-2}$ is zero,
where $\{S_k\}_{k\geq 0}$ is the sequence defined recursively by
$S_0 = 4$ and $S_k = \Bigl[\bigl(S_{k-1}^2-2\bigr)
  \operatorname{\textsf{\upshape\small{}modulo}} M(p)\Bigr]$ for $k>0$.
\end{LehmerTheorem}

Notice that checking the condition stated in the above theorem requires to
compute squares and modulus of arbitrarily large numbers.
In any event, to prove for instance that $M(7)=2^7-1=127$ is a prime number
it is enough to prove that $S_5$ is zero.
Easy computations show that
\begin{align*}
S_0 & = 4                  & S_3 & = 4487\phantom{9}\pmod{127} = 42\\
S_1 & = 14                 & S_4 & = 1762\phantom{9}\pmod{127} = 111\\
S_2 & = 194\pmod{127} = 67 & S_5 & = 12319\pmod{127} = 0;
\end{align*}
thus concluding that $M(7)$ is prime.

In this note we propose an alternative approach
to the determination of the primality of large Mersenne numbers
based on the study of the \emph{circle doubling map}
$\delta(\theta)=2 \theta \pmod{1}$.
This map is the natural model for the expanding covering maps of degree two defined on the unit circle $\SI$, whose iterates defined a chaotic discrete dynamical system (see Section~\ref{SEC::DoublingMap} for precise definitions, and \cite{OneDimensionalDynamics} as an excellent reference over the whole theory).

The theoretical result underlying our algorithms for determining the primality of Mersenne numbers with prime exponent admits a simple formulation, yet establishes a remarkable connection between the dynamics of $\delta$ and the arithmetic structure of Mersenne numbers. More precisely, it relates the fixed points of $\delta^n$ to the divisors of $M(n)$, whenever such divisors exist.

In what follows, we denote by $\floor{\cdot}$ the integer part function.

\begin{MainTheorem}\label{MainTheorem::A}
Let $n>1$ be a natural number.
\begin{enumerate}[(a)]
\item An odd number
      $q \in \bigl\{3,5,7,\dots,M(n)-2\bigr\}$ is a divisor of
      $M(n)$ if and only if $\frac{1}{q} \in \SI$ is a fixed point of $\delta^n$.
\item The Mersenne number $M(n)$ is prime if and only if the exponent $n$ is prime and, for every $q \in \biggl\{3,5,7,\dots,\elasticfloor{\sqrt{M(n)}}\biggr\}$ odd, $\frac{1}{q} \in \SI$ is $\delta-$periodic with period $k_q\ne n$.
\end{enumerate}
\end{MainTheorem}

In the particular case where $n$ is prime, every fixed point of $\delta^n$ necessarily corresponds to a periodic point of $\delta$ with exact period $n$. This observation is fundamental for our approach, since it allows the primality problem for Mersenne numbers to be reformulated in terms of the periodic dynamics of the map $\delta$.
More precisely, if $n$ is prime then Theorem~\ref{MainTheorem::A}(a) is equivalent to the statement
\begingroup
\itshape
\begin{enumerate}[(a')]
\item Assume $n$ is prime. An odd number $q \in \bigl\{3,5,7,\dots,M(n)-2\bigr\}$ is a divisor of  $M(n)$ if and only if $\frac{1}{q} \in \SI$ is $\delta-$periodic of period $n$.
\end{enumerate}
\endgroup

\begin{remark}\label{rem:period_M_n}
If in Theorem~\ref{MainTheorem::A}(a) we take $q = M(n)$, we have
\[
2^n \tfrac{1}{q} \pmod{1} =
   (2^n-1+1)\frac{1}{2^n-1} \pmod{1} =
   \left(1 + \frac{1}{2^n-1}\right)\pmod{1} =
   \frac{1}{2^n-1}.
\]
In other words, $\frac{1}{M(n)}$ is $\delta-$periodic of period $n$, for all $n > 1$.
This justifies the assumption that, in Theorem~\ref{MainTheorem::A}(a),
$q$ must be smaller than $M(n).$
\end{remark}

Beyond our theorem, the relationship between Mersenne numbers and dynamical concepts has been explored extensively in the literature. For example, in \cite{MalWhi}, the authors employed classical symbolic dynamics associated with the angle-doubling map over the $p$-adic numbers, for $p$ prime, to investigate the primality of $M(p)$. Likewise, \cite{KHAA} introduced a generalization of Mersenne numbers, known as $k$-Mersenne numbers, in connection with the study of certain difference equations. Our approach differs from these perspectives in that it relates the arithmetic properties of Mersenne numbers directly to the periodic dynamics of the map $\delta$, thereby providing a new dynamical framework for addressing questions of primality and divisibility.

In light of Theorem~A, we have been able to develop a collection of algorithms for investigating the primality (or not) of $M(n)$ through the computation of periodic points of $\delta$, without requiring the explicit evaluation of the Mersenne number $M(n)$ itself. Theorem~A is particularly valuable in this context because it transforms the arithmetic problem of determining the primality of $M(n)$ into a dynamical problem involving the efficient computation of $\delta$-periods.

Since these period computations can be implemented very efficiently,
the resulting algorithms lead to non-trivial and computationally effective
programs for testing the (non) primality of Mersenne numbers with prime exponent.
Furthermore, the flexibility of the method suggests the possibility of additional
refinements and optimisations.
For these reasons, we discuss the algorithms and their implementation
in detail in Section~\ref{SEC::Algos}.

Moreover, the complete collection of programs, written in the C programming language,
together with the corresponding computational outputs, is freely available in the public repository
{\color{blue}\url{https://github.com/CombTopDynamics-cat/Mersenne-doubling.git}}

The paper is organised as follows.
In Section~\ref{SEC::DoublingMap}, we introduce the notion of discrete dynamical system,
and summarise several properties of the doubling map on the unit circle.
In Section~\ref{section:theoremAB}, we prove Theorem~\ref{MainTheorem::A},
and discuss their theoretical implications.
In Section~\ref{SEC::Algos}, we present algorithms to compute periods under the doubling map
of periodic points of the form $\frac{1}{q} \in \SI,$ where $q$ is an odd integer.
We also analyse the efficiency of these algorithms and present a first application
to find composite Mersenne numbers for huge values of $n$.
Finally, in Section~\ref{SEC::Prime_Mersenne_numbers}, we present a primality test for Mersenne numbers based on
Theorem~\ref{MainTheorem::A}, and the algorithms introduced in Section~\ref{SEC::Algos}.

\section{The doubling map and its periodic points}\label{SEC::DoublingMap}

A \emph{discrete dynamical system} $(f,M)$
can be defined as the action of a continuous function $\map{f}{M}$
on a topological space $M$.
We consider the sequence $\bigl\{f^k\bigr\}_{k \in \N}$ of iterates of $f$,
where $f^k$ denotes the composition of $f$ with itself $k$ times.
We say that a point $z \in M$ is \emph{periodic of (minimal) period $n \geq 1$} if $f^n(z) = z$
and $f^\ell(z) \neq z$ for every $\ell = 1,2,\dots, n-1.$ Similarly,  we say that $x \in M$ is
\emph{pre-periodic} if it is not periodic but
$f^{k}(x)$ is periodic for some $k\geq 1$.

\begin{remark}\label{Remark:periodic}
The set of all solutions of $f^m(z) = z, \ z\in M$ determines the
\emph{set of fixed points of $f^m$}, $m\in \mathbb N$.
Not every fixed point of $f^m$ is a periodic point of period $m$ (defined above).
In what follows, we assume that the period is always \emph{minimal}
(without explicit mention), and so it is uniquely defined.
For brevity, if $z\in M$ is periodic for $f$ of period $n$
we will say that $z$ is \emph{$f-$periodic of period $n$},
or that $n$ is the $\delta-$period of the (periodic point) $z$.

In this framework it is worth noticing that every
$f-$periodic point of period $n$ is a fixed point of
$f^{kn}$ for every $k \in \N.$
\end{remark}

The sequence of all iterates for a given $z \in M$ is called the \emph{$f-$orbit of $z$} and is denoted by
\[
\operatorname{Orb}_f(z) := \bigl\{z,f(z),\dots,f^k(z),\ldots\bigr\}.
\]

A particular case of a discrete dynamical system is the pair $(\delta,\SI)$,
where $\delta$ is the \emph{doubling map} and $\SI$ is the unit circle.
More precisely, we define the unit circle $\SI$ as $\R / \Z \equiv [0,1)$,
and the doubling map $\map{\delta}{\SI}$ is defined by $\delta(\theta) := 2\theta \pmod{1}$.
Hence, for every $\theta \in \SI$, the $\delta-$orbit of $\theta$ is
\[
    \operatorname{Orb}_{\delta}(\theta) = \elasticset{2^k \theta \pmod{1}}{k = 0,1,2,\ldots}.
\]
Of particular interest are the periodic points of period $n\geq 1$ for $\delta$. Those points must be solutions of the equation
\[
    (2^n-1) \theta \equiv 0 \,  \pmod{1}.
\]
determining all fixed points of $\delta^n$ (see Remark \ref{Remark:periodic}).  Equivalently,
\begin{equation}\label{eq:per_points}
     \biggl\{\frac{1}{2^n-1}, \frac{2}{2^n-1},\dots, \frac{2^n - 2}{2^n-1}\biggr\}
\end{equation}
In particular, the set of periodic points of the doubling map is a dense subset of $\SI$, but there are no periodic points of a given period $n>1$ in the arc $\Bigl[0,\tfrac{1}{2^n-1}\Bigr)$.

\begin{remark}
The unique fixed point of the doubling map is $0$, whereas
$\tfrac{1}{3}$ and $\tfrac{2}{3}$ are the only periodic points of period $2$.
In fact,
$\delta\bigl(\tfrac{1}{3}\bigr) = \tfrac{2}{3}$ and $\delta\bigl(\tfrac{2}{3}\bigr) = \tfrac{1}{3}$
so that $\Bigl\{\tfrac{1}{3}, \tfrac{2}{3}\Bigr\}$ is the periodic orbit of period 2.
\end{remark}

\begin{remark}\label{REM::NoEsOroTodoLoQueReluce}
A rational angle $\frac{1}{q}$ with $q \geq 3$ odd is not necessarily of the form $\frac{1}{2^n -1}$ in
\eqref{eq:per_points}.
Indeed, not every odd $q$ is of the form $2^n -1$ for some $n$. As an example it is easy to check that this is the case for $q = 5$, and
that the rational angle $\tfrac{1}{5}$ has period 4 because
\begin{align*}
    \delta\bigl(\tfrac{1}{5}\bigr) &= \tfrac{2}{5}, &
    \delta\bigl(\tfrac{2}{5}\bigr) &= \tfrac{4}{5}, \\
    \delta\bigl(\tfrac{4}{5}\bigr) &= \tfrac{8}{5}\text{\footnotesize$\pmod{1}$} = \tfrac{3}{5}, &
    \delta\bigl(\tfrac{3}{5}\bigr) &= \tfrac{6}{5}\text{\footnotesize$\pmod{1}$} = \tfrac{1}{5}.\\
\end{align*}
Therefore, in \eqref{eq:per_points},
\[
   \frac{1}{5} = \frac{3}{2^4 -1} \neq \frac{1}{2^4 -1}.
\]

On the other hand, from Theorem~\ref{MainTheorem::A}(a),
$5$ is a divisor of $M(4) = 2^4 - 1 = 5 \times 3$.
\end{remark}

The key tool for proving Theorem~\ref{MainTheorem::A} is
to connect the dynamics of the rational angles in $\SI$ under the doubling map $\delta$
to Mersenne numbers.
Notice that \eqref{eq:per_points} provides a natural link
through which this connection occurs. The following technical lemma, which is a particular case of Theorem~5.1 in \cite{CarlesonGamelinBook}, will be used  in the  proof of Theorem~\ref{MainTheorem::A}.

If $a,b \in \N$ we denote by $\gcd(a,b)$ the \emph{greatest common divisor} of $a$ and $b$. Notice that if either $a$ or $b$ are prime numbers then $\gcd(a,b) = 1$.

\begin{lemma}\label{LEM::CarlesonGamelin}
Let $q \in \N$ and let $p \in \{0,1,\dots,q-1\}$ be with $\gcd(p,q) = 1$.
Then, the following statements hold.
\begin{enumerate}[(a)]
\item The point $\tfrac{p}{q} \in \SI$ is $\delta-$pre-periodic if and only if the denominator $q$ is even.
\item The point $\tfrac{p}{q} \in \SI$ is $\delta-$periodic if and only if the denominator $q$ is odd.
\end{enumerate}
\end{lemma}

\section{Proof of Theorem~\ref{MainTheorem::A}}\label{section:theoremAB}

In this section we present the proof of Theorem~\ref{MainTheorem::A},
and discuss several theoretical implications.

\begin{proof}[Proof of Theorem~\ref{MainTheorem::A}]
We start by proving the part (a).
Let $q \in \N$ odd and assume that $\frac{1}{q}\in \mathbb S^1$ is $\delta-$periodic  of period $n$ (in particular it is a fixed point of $\delta^n$). Then we have  $2^n  \equiv 1 \, \pmod{q}$ or, equivalently,
\begin{equation}\label{eq:alpha}
2^n -1 = q \alpha
\qquad\text{for some $\alpha \in \N$.}
\end{equation}
Therefore, $q$ divides $M(n) = 2^n-1$.

On the other hand, if $q \in \N$ is an odd divisor of $M(n) = 2^n-1$ for some $n$,
then \eqref{eq:alpha} holds again and, equivalently,
\[
\frac{1}{q} = \frac{\alpha}{2^n-1}.
\]
So, from \eqref{eq:per_points} (see also Remark~\ref{REM::NoEsOroTodoLoQueReluce}),
$\frac{1}{q}$ is a fixed point of $\delta^n$. This ends the proof of Statement~(a).

To prove Statement~(b) we just notice that the primality of $M(n)$ is equivalent to
the non-existence of odd (prime) divisors of $M(n)$.
Moreover, by Lemma~\ref{LEM::CarlesonGamelin},
$\frac{1}{q}$ is $\delta-$periodic for every $q \in \biggl\{3,5,7,\dots,\elasticfloor{\sqrt{M(n)}}\biggr\}$ odd. Then Statement~(b) follows from Statement~(a).
\end{proof}

We know that if $n$ is composed then $M(n)$ is also composed. Next result complements this statement for the case when $n$ is even using the doubling map.

\begin{corollary}
Assume that $\theta=\frac{1}{q}$ is $\delta-$periodic of period $n$.
Then $q$ divides $M(n\ell)$ for all $\ell>0$.
In particular, 3 divides $M(n)$ for all $n$ even.
\end{corollary}

\begin{proof}
If $\frac{1}{q}\in \SI$ is $\delta-$periodic of period $n$,
then clearly $\delta^{n\ell}\bigl(\frac{1}{q}\bigr)=\frac{1}{q}$
and so $\frac{1}{q}$ is a fixed point of $\delta^{n\ell}$, for all $\ell\geq 1$.
Hence, Theorem~\ref{MainTheorem::A}(a) implies that $q\mid M(n\ell)$ for all $\ell>0$.
Moreover, since
\[
 \frac{1}{3} \longmapsto \frac{2}{3} \longmapsto \frac{4}{3} \equiv \frac{1}{3} \pmod{1}.
\]
we conclude that $\frac{1}{3}$ has minimal period 2 and so $3\mid M(2\ell)$ for all $\ell>0$.
  \end{proof}

We conclude this section by interpreting the main conjectures concerning Mersenne numbers through the dynamics of the doubling map. As discussed in the introduction, one of the central open problems in the theory of Mersenne numbers is whether the set of prime Mersenne numbers is finite or infinite. Recasting this question in dynamical terms provides a different perspective on the problem and suggests new avenues for its investigation.

\begin{conj}\label{conj:prime_infinite}
The sequence of prime Mersenne numbers is infinite.
\end{conj}

In this context there are many conjectures relating  the primality of $M(n)$ with the \emph{form} of the prime number $n$.  In the early works of Mersenne it was implicit that he considered the following statement as valid (see \cite{BatSelWag}): {\it
$M(n)$ is prime if and only if $n$ is of the form $2^k\pm 1$ or $2^{2k}\pm 3$}. Going forward on the eventual characterization of Mersenne primer numbers the authors in \cite{BatSelWag} propose the following conjecture.

\begin{conj}
If two of the following statements about an odd positive integer $n$ are true, then the third one is also true.
\begin{itemize}
\item[(a)] $n=2^k\pm 1$ or $n=4^k\pm 1$ (for some $k>1$)
\item[(b)] $M(n)$ is prime
\item[(c)] $(2^n+1)/3$ is prime
\end{itemize}
\end{conj}

To the best of our knowledge, all known results or conjectures concerning the primality of $M(n)$ and the finiteness of the sequence of Mersenne primes are number-theoretic in nature. Theorem~\ref{MainTheorem::A} allows to write Conjecture  \ref{conj:prime_infinite} in terms of the dynamics of the period doubling.

\begin{conj}
For all $n_0>0$ there exists a primer number $n>n_0$ such that for all $q$ prime with $1<q\leq \elasticfloor{\sqrt{M(n)}}$ the angles $\frac{1}{q}$ have period $k_q\ne n$ under the doubling map $\delta$.
\end{conj}

\section{Efficient algorithms to compute periods of the doubling map}\label{SEC::Algos}

 In this section, we develop the machinery needed to apply Theorem~\ref{MainTheorem::A}(a) to the detection of composite Mersenne numbers. To this end, we assume throughout  that $q\in \mathbb N$ is odd, and we construct an efficient algorithm for computing the $\delta$-period of $\frac{1}{q} \in \SI$. This algorithm provides a practical means of relating the dynamics of $\delta$ to the arithmetic properties of Mersenne numbers.

 By Lemma~\ref{LEM::CarlesonGamelin}, $\frac{1}{q}\in \SI$  is $\delta$-periodic. Moreover, Theorem~\ref{MainTheorem::A}(a) shows that if $\frac{1}{q} \in \SI$ is $\delta-$periodic of period $n$ (and so, it is a fixed point of $\delta^n$) then $q \mid M(n)$. Thus, determining the $\delta$-period of $\frac{1}{q}$ provides direct arithmetic information about the divisibility properties of the Mersenne number $M(n)$, and in particular offers a method for detecting its compositeness.

In the first subsection we elaborate the algorithms for computing the $\delta-$periods and discuss their efficiency.
In the second subsection we present and discuss the output of these algorithms
by showing a large set of natural numbers $q$ for which $\frac{1}{q}\in \SI$ is $\delta-$periodic of period $n_q$.
The interest of this output is that many of the $n_q$-periods would be huge and prime,
so the Mersenne numbers $M(n_q)$ would be composite,
corresponding to $n_q$ prime and further beyond the 52th Mersenne number.

\subsection{The algorithms}
Since $\delta-$periodic points of period $n$ are given by  \eqref{eq:per_points} we conclude that if
$\frac{1}{q}\in \SI$ is $\delta-$periodic of period $n$ then $2^n  \equiv 1 \pmod{q}$.
In this spirit we define the \emph{$q-$integer doubling map} $\Delta_q$ as follows:
\begin{definition} \label{DEF::IntegerDoubling}
Let $q\in \N$ (odd) with $q \geq 3$.
Set $\ZZ_q := \{1,2,\dots,q-1\}$ and define
$\map{\Delta_q}{\ZZ_q}$ by $\Delta_q(r) := 2r\pmod{q}.$
\end{definition}

\begin{remark}
Observe that, since $q$ is odd, we have $0<\Delta_q(r)\leq q-1$ for every $r \in \ZZ_q$ and,
hence, $\Delta_q$ is well defined.
\end{remark}

By using the \emph{$q-$integer doubling map} $\Delta_q,$
we can re-formulate the $\delta-$periodicity of $\frac{1}{q}$.
\begin{lemma}\label{LEM::Basic}
Let $q\in \N$ be odd with $q \geq 3$.
Then, $\frac{1}{q}\in \SI$ has $\delta-$period $n$ if and only if
$1 \in \ZZ_q$ has $\Delta_q-$period $n$.
\end{lemma}

Note that iterating the integer doubling map $\Delta_q$ consists of multiplying by two
$$
r \mapsto 2r=r+r
$$
and modulo $q$ the result.
From the point of view of efficiency, multiplying by two
is arithmetically negligible but the modulo operation is something to be avoided.
Certainly, in all cases for which $2r < q$ the $\pmod{q}$ operation is not necessary.
To \emph{predict} when it will be necessary to compute modulo $q$
(after the operation $r \mapsto 2r=r+r$)
when iterating $\Delta_q$ we introduce the \emph{$q-$Poincar\'e integer doubling map}.

\begin{definition} \label{DEF::qPoincare}
The \emph{$q-$Poincar\'e integer doubling map} $\map{\pi_q}{\ZZ_q}$ is defined by
\begin{equation}\label{EQ::flying}
\pi_q(r) := 2^{p} \cdot r - q,
\end{equation}
\noindent where $p \in \Z^+$ is the smallest positive integer such that $2^{p} \cdot r \geq q.$
The number $p$ is called the \emph{flying time of $\pi_q(r)$}.
Observe that the flying time $p$ always exists and it is positive because $r \in \ZZ_q.$
\end{definition}

\begin{remark}\label{REM::FlyingTimeBounds}
Let $\ell >0$. The minimality condition on the number $p$ in the above definition implies that
\begin{enumerate}[(i)]
\item if $2^\ell \cdot r \geq q$ then the \emph{flying time of $\pi_q(r)$} is smaller than or equals to $\ell$;
\item if $2^{\ell} \cdot r < q$ then the \emph{flying time of $\pi_q(r)$} is at least $\ell + 1$, and
\item if $2^\ell > q$ then the \emph{flying time of $\pi_q(r)$} is smaller than or equal to $\ell$, for all $r\in\ZZ_q$.
\end{enumerate}
\end{remark}

The following proposition is an immediate consequence of
Lemma~\ref{LEM::CarlesonGamelin}, Lemma~\ref{LEM::Basic}, and the preceding definitions.
It is the key tool to compute periods with the help of the
$q-$Poincar\'e integer doubling map $\pi_q$.

\begin{prop}\label{PROP::Flying}
The following statements hold for every $q\geq 3$ odd.
\begin{enumerate}[(a)]
\item $\frac{1}{q}\in \SI$ is $\delta-$periodic.
\item $1 \in \ZZ_q$ is $\Delta_q-$periodic.
\item The $\delta-$period of $\frac{1}{q}\in \SI$ and
      the $\Delta_q-$period of $1 \in \ZZ_q$
      coincide.
\item $1 \in \ZZ_q$ is $\pi_q-$periodic.
\item The $\Delta_q-$period of $1 \in \ZZ_q$
      is the sum of the flying times of the set
\[
   \Bigl\{\pi_q(1), \ \pi_q\bigl(\pi_q(1)\bigr), \ \pi_q\bigl(\pi^2_q(1)\bigr),\dots,\pi_q\bigl(\pi^{m-1}_q(1)\bigr)\Bigr\},
\]
where $m\in\N$ denotes the $\pi_q$-period of $1\in \ZZ_q$.
\end{enumerate}
\end{prop}

Algorithm~\ref{ALGO::Naive} below arises from the last statement of Proposition~\ref{PROP::Flying}
and is a direct, naive method to compute periods of points
$\frac{1}{q}\in \SI$ under the circle doubling map $\delta$.
Observe that we already know that the period of $\frac{1}{3}\in \SI$ is $2$,
and therefore we are not interested in algorithms to compute this period.

\begin{algorithm}[ht]
\caption{\hspace*{1em}\parbox{32em}{%
Computation of the period of $\tfrac{1}{q}$ by the direct, naive use of Proposition~\ref{PROP::Flying}}}\label{ALGO::Naive}
\algdef{SE}[DOWHILE]{Do}{EndDo}{\algorithmicdo}[1]{\algorithmicwhile\ #1}%
\begin{algorithmic}
\Ensure $q \geq 5$ odd{\color{blue}\Comment{$q \geq 5$ implies $qh := \tfrac{q-1}{2} > 1$}}
\Procedure{\_\_period\_of\_one\_divided\_by}{$q$}
\State $qh \gets \tfrac{q-1}{2}${\color{blue}\Comment{\footnotesize{}Initialising global auxiliary variable}}
\State $r \gets 1${\color{blue}\Comment{\footnotesize{}Initialising iteration}}
\State $\text{\sffamily{}period} \gets 0$
    \Do
        \State $ft \gets 0${\color{blue}\Comment{\footnotesize{}Initialising iteration of $\pi_q$}}
        \While{$r \leq qh$}
               \State $r \gets r + r$
               \State $ft \gets ft + 1$
        \EndWhile
        \State $r \gets r - (q - r)${\color{blue}\Comment{\footnotesize{}Computing $2r\pmod{q}$ safely (to avoid overflows) and efficiently}}
        \State $\text{\sffamily{}period} \gets \text{\sffamily{}period} + ft + 1$
    \EndDo{$r \neq 1$}
    \State \Return \text{\sffamily{}period}
\EndProcedure
\end{algorithmic}
\end{algorithm}

Remember that if  $a \in \R^+$ then $\floor{a}$ denotes the largest
non-negative integer which is smaller than or equals to $a$.
Moreover,  $\log_2$  denotes the logarithm in base $2$.

A serious drawback of Algorithm~\ref{ALGO::Naive} is that the flying times are,
in fact, computed by trial and error.
To overcome this problem we will explore a \emph{predictive strategy}.
More precisely, recall that for every $r \in \ZZ_q$,
the flying time of $\pi_q(r)$ is the smallest positive integer $p$ such that $2^{p} \cdot r \geq q.$
Since $2^{p} \cdot r$ is even and $q$ is odd, this inequality is equivalent to
$2^{p} \cdot r > q-1$ and, since $p > 0$, the minimality of $p$ gives
$2^{p} \cdot r > q-1 \geq 2^{p-1} \cdot r.$
Since $2^{p-1} \in \N,$ this is equivalent to
\[
   2^{p} > \frac{q-1}{r} \geq \elasticfloor{\frac{q-1}{r}} \geq 2^{p-1},
\]
which in turn is equivalent to
\[
   p > \log_2\biggl(\elasticfloor{\frac{q-1}{r}}\biggr) \geq p-1.
\]
Consequently,
\begin{equation*} 
   p-1 = \elasticfloor{\log_2\biggl(\elasticfloor{\frac{q-1}{r}}\biggr)}.
\end{equation*}
This gives the {\it predictive} Algorithm~\ref{ALG::PredictivePoincare} to compute one iterate of the $q-$Poincar\'e integer doubling map $\pi_q$.
In this algorithm we denote $p-1$ by $t$ and, equivalently, $p := t+1.$
Concerning the computation of $\pi_q(r) := 2^{p} \cdot r - q,$
recall that $2^{t+1} \cdot r \geq q+1 > q-1 \geq 2^{t} \cdot r.$

\begin{algorithm}
\caption{Iterating the $q-$Poincar\'e integer doubling map with a predictive algorithm}\label{ALG::PredictivePoincare}
\begin{algorithmic}
\Procedure{predictive\_Poincare\_doubling\_map}{$r$, $q$}
  \State $t \gets \elasticfloor{\log_2\Bigl(\elasticfloor{\frac{q-1}{r}}\Bigr)}$
\vspace*{4pt}
  \State $a \gets r\cdot 2^{t}$
  \State \textbf{return} $\bigl[a - (q-a), t + 1\bigr]$
\EndProcedure
\end{algorithmic}
\end{algorithm}

Notice that the computational cost of the predictive algorithm of the
$q-$Poincar\'e integer doubling map is independent on $r$.
It strongly depends on the implementation (and the computational cost)
of evaluating $\floor{\log_2(\cdot)}$.
In contrast, the computational cost of the non predictive evaluation
of an iterate of the $q-$Poincaré integer doubling map depends
on the flying time that, in turn, depends on $r$.

For small flying times, the non-predictive algorithm must perform better;
it replaces the costly computation of $\floor{\log_2(\cdot)}$ with
a small number of additions and comparisons.
In contrast, for large flying times, the predictive algorithm avoids computing
a large number of additions and, specially, \emph{comparisons}
by performing only a single $\floor{\log_2(\cdot)}$ evaluation.

The efficient algorithm to compute iterates of the
$q-$Poincar\'e integer doubling map $\pi_q(r)$
relies on deciding a priori whether $r$ is small enough to deserve
a predictive evaluation of the flying time by using $\floor{\log_2(\cdot)}$,
or rather $r$ is too large for that and we have to use the non-predictive algorithm.
However, observe that if $r$ is large (whatever it means) it is contradictory
to estimate a priori the flying time without knowing the benefits of doing it.

Next we will algorithmically try to overcome the problem of finding a priori bounds
for the flying times.

\begin{definition} \label{DEF::kappa}
We define the \emph{critical efficiency boundary for the flying times},
denoted by $\kappa$,
to be the positive integer such that for flying times smaller than or equals to $\kappa$
the non predictive algorithm to compute one iterate of the
$q-$Poincar\'e integer doubling map $\pi_q(r)$ is the winner
(meaning that is more efficient that the predictive one).
\end{definition}

\begin{remark}
From Remark~\ref{REM::FlyingTimeBounds} we get
$
   \kappa \in \{1,2,\dots,63,64\},
$
when the integers $q$ are unsigned integers of 64 bits.
Of course, the precise value of $\kappa$ depends on the specific implementation of both algorithms,
and in particular on the implementation of $\floor{\log_2(\cdot)}$.
We will discuss later the choice of $\kappa$ used in our computations.
\end{remark}

The next lemma gives an easy algorithm (in terms of the
\emph{critical efficiency boundary for the flying times} $\kappa$)
to a priori choose between the predictive
and the non-predictive algorithms when computing $\pi_q(r)$.
This is possible because we are using $\kappa$ as the discriminator between
\emph{small} and \emph{large} flying times.
In this lemma we have to require $q \geq 5$
(equivalently $\tfrac{q-1}{2} > 1$)
because this was a working requirement for
Algorithm~\ref{ALGO::Naive}.

\begin{lemma}\label{LEM::qkappabounds}
Let $\kappa \in \{1,2,\dots,64\}$, and let
$q \geq \max\bigl\{5, 2^{\kappa}\bigr\}$ odd.
Set $\sigma(q) := \floor{\log_2(q)}$. Then, the following statements hold.
\begin{enumerate}[(a)]
\item $2^{\sigma(q)} < q < 2^{\sigma(q)+1}$.
\item The flying time of $\pi_q(1)$ is exactly $\sigma(q) + 1$.
\item The predictive algorithm to compute $\pi_q(r)$ is more efficient than
      the non-predictive one if and only if $r \in \bigl\{1,2,\dots,\elasticfloor{\frac{q}{2^{\kappa}}}\bigr\}.$
\end{enumerate}
\end{lemma}

\begin{remark}
The assumption $q \geq 2^{\kappa}$ in the above lemma implies that
$\elasticfloor{\frac{q}{2^{\kappa}}} \geq 1$ and, hence,
$\pi_q(1)$ always must be computed with the predictive algorithm.

On the other hand, the fact that $\kappa \leq 64$ is a consequence
of the (implicit) assumption that $q$
is an unsigned integer of 64 bits.
Equivalently, $q \leq 2^{64} -1$ and, consequently,
$q \geq 2^{\kappa}$ implies $k < 64.$
\end{remark}

\begin{proof}[Proof of Lemma~\ref{LEM::qkappabounds}]
Let $p\in \N$ be the largest value such that $2^p\leq q$.
Since $q$ is odd this implies that $p$ is the largest positive integer such that $2^p < q$ or, equivalently, the largest positive integer such that $p < \log_2(q)$.
Since $q \geq 5,$ $p = \floor{\log_2(q)} = \sigma(q).$
This proves Statement~(a).

Statement~(b) follows directly from Statement~(a) and Definition~\ref{DEF::qPoincare}.

Now we prove Statement (c).
Take $r \in \bigl\{1,2,\dots,\elasticfloor{\frac{q}{2^{\kappa}}}\bigr\}.$
Taking into account that $q$ is odd we have
\[
   2^{\kappa}\cdot r \leq
       2^{\kappa} \elasticfloor{\frac{q}{2^{\kappa}}} <
       2^{\kappa} \frac{q}{2^{\kappa}} = q.
\]
Hence (see Remark~\ref{REM::FlyingTimeBounds}), the flying time of $\pi_q(r)$ is at least $\kappa + 1$
and, by Definition~\ref{DEF::kappa}, the predictive algorithm to compute $\pi_q(r)$ is more efficient than
the non-predictive one.
Conversely, assume that $r > \elasticfloor{\frac{q}{2^{\kappa}}}.$
Since $r$ is an integer we have
\[
   2^{\kappa}\cdot r > 2^{\kappa}\frac{q}{2^{\kappa}} = q.
\]
Again by Remark~\ref{REM::FlyingTimeBounds}
the flying time of $\pi_q(r)$ is at most $\kappa$,
and the non-predictive algorithm is more efficient than the predictive one.
\end{proof}


Lemma~\ref{LEM::qkappabounds} and Definition~\ref{DEF::kappa}
give a new strategy (Algorithm~\ref{ALGO::ElBo})
for computing periods of the doubling map.
It consists of determining
the best algorithm (predictive or naive) to compute
an iteration of the $q-$Poincar\'e integer doubling map $\pi_q(r)$
in terms of $r$ and the critical efficiency limit for the flying time $\kappa$.

\begin{algorithm}
\caption{\hspace*{1em}\parbox{32em}{%
Efficient combination of the predictive and non predictive algorithms
to compute the period of $\frac{1}{q}$ by the doubling map for large values of $q$
\\[4pt]
It uses the function \textsc{\_floor\_log2} which is an efficient implementation
of the evaluation of $\floor{\log_2(\cdot)}$}}\label{ALGO::ElBo}
\algdef{SE}[DOWHILE]{Do}{EndDo}{\algorithmicdo}[1]{\algorithmicwhile\ #1}%
\begin{algorithmic}
\item[\textbf{Define:}] $\kappa\in \{1,2,3,\dots,63,64\}$ as internal global parameter
\Ensure $q \geq \max\bigl\{5, 2^{\kappa}\bigr\}$ odd
\Procedure{\_\_efficient\_computation\_of\_the\_period\_of\_one\_divided\_by}{$q$}
\State $qh \gets \tfrac{q-1}{2}${\color{blue}\Comment{\footnotesize{}Initialisation of global variables}}
\State $r_{\textsf{\tiny boundary}} \gets \elasticfloor{\frac{q}{2^{\kappa}}}${\color{blue}\Comment{\parbox[t]{28em}{\footnotesize{}
                                                                               Discriminator bound between predictive and non-predictive algorithms
                                                                               to compute the iterates of the $q-$Poincar\'e integer doubling map
                                                                               (Lemma~\ref{LEM::qkappabounds}(c))}}}
\State $t \gets \Call{\_floor\_log2}{q}${\color{blue}\Comment{\parbox[t]{18em}{\footnotesize{}Initialising the $q$-Poincaré doubling map iteration
                                                                                    by computing $r = \pi_q(1)$ (see Lemma~\ref{LEM::qkappabounds}(b))}}}
\vspace*{-1.5ex}
\State $a \gets 2^{t}$
\State $r \gets a - (q - a)$
\State $\text{\sffamily{}period} \gets t + 1$
\Statex
\While{$r \neq 1$}
   \If{$r > r_\textsf{\tiny boundary}$}{\color{blue}\Comment{\footnotesize{}Choosing between the predictive and non-predictive algorithms}}
       \State $t \gets 0${\color{blue}\Comment{\parbox[t]{13em}{\footnotesize{}Preparing the computation of $\pi_q(r)$ by the non-predictive algorithm}}}
       \vspace*{-1.5ex}
       \State $a \gets r$
       \While{$a \leq qh$}
          \State $a \gets a + a$
          \State $t \gets t + 1$
       \EndWhile
   \Else
       \State $t \gets \textsc{\_floor\_log2}\Bigl(\elasticfloor{\tfrac{q-1}{r}}\Bigr)${\color{blue}\Comment{\parbox[t]{13em}{\footnotesize{}Preparing
                                                                                           the computation of $\pi_q(r)$ by the predictive algorithm}}}
       \vspace*{-1.5ex}
       \State $a \gets r\cdot 2^{t}$
   \EndIf
   \Statex\vspace*{-2ex}
   \State $r \gets a - (q - a)${\color{blue}\Comment{\footnotesize{}The effective computation of $\pi_q(r)$}}
   \State $\text{\sffamily{}period} \gets \text{\sffamily{}period} + t + 1$
\EndWhile
\State \Return \text{\sffamily{}period}
\EndProcedure
\end{algorithmic}
\end{algorithm}

The above algorithm is implemented in a program in C language called

\centerline{\texttt{\color{blue}compute\_all\_periods\_in\_a\_range\_of\_odd\_qs.c}}

\noindent which is freely available at the public repository

\centerline{\color{blue}\url{https://github.com/CombTopDynamics-cat/Mersenne-doubling.git}}

The program \texttt{\color{blue}compute\_all\_periods\_in\_a\_range\_of\_odd\_qs.c}
is nothing else than Algorithm~\ref{ALGO::ElBo}
after checking that both endpoints of the interval of $q$'s verify
$q \geq \max\bigl\{5, 2^{\kappa}\bigr\}$.
Moreover, it is adaptative in the sense that if the first endpoint of the
interval of $q$'s is larger than the second one,
then the computation is done downwards from the first endpoint to the second one.

The code of the program \texttt{\color{blue}compute\_all\_periods\_in\_a\_range\_of\_odd\_qs.c}
includes our (very efficient --- we believe) implementation of the function
\texttt{\color{blue}\_floor\_log2} that computes $\floor{\log_2(\cdot)}$ for unsigned integers of 64 bits.
Computing $\floor{\log_2(u)}$ is equivalent to determining the
largest exponent of a power of two that is smaller than or equals to $u$.
This, in turn, is equivalent to determining the position of the dominant one in the binary
expansion of $u$. Thus, our strategy to compute $\floor{\log_2(u)}$ consists in looking at
the binary expansion of $u$ and determining the position of the dominant one.
The implementation is based on an auxiliary table of 256 entries, that gives
the positions of the dominant ones for all non-negative integers from 0 to 255
(that is, for all non-negative integers that fit in one byte).
Then, the determination of the position of the leftmost one in the binary expansion of $u$ is done
by using at most three bisection steps to find the byte of $u$ that contains the most
significant binary digit.
Finally, the position of this leftmost binary digit of $u$ is obtained after
some accounting with the help of the table mentioned before.

Next, we explain how the value of $\kappa$ is determined.
Indeed, we select the value of $\kappa$ empirically by optimising
the performance of the algorithm.
To carry out this optimisation, we require odd integers $q$
for which the computation of the period of $\tfrac{1}{q}$
involves a large number of iterates of the
$q$-Poincaré integer doubling map, together with
a wide statistical distribution of flying times.

Notice that the above wishful thinking it is not general at all.
Indeed, it is easy to see that the points $\frac{1}{q} \in \SI$
with $q = 2^p - 1$ have period $p$,
and \emph{this period is computed in a single iterate of the
$q-$Poincaré integer doubling map which has flying time $p$}.
In contrast, at the other end of the complexity spectrum
we can find odd values of $q$
--- that we call \emph{complete with respect to flying times} ---
such that $2^{s-1} < q < 2^s$ for some $s \in \N$,
and for each $t \in \{1,2,\dots,s\}$ there is (at least) one iterate
$\pi_q(r)$ of the $q-$Poincaré integer doubling map that has flying time exactly $t$.
Curiously enough, a normal behaviour of the
\emph{complete numbers with respect to flying times}
is that they verify a kind of Zipf's law, probably associated to the (binary)
itinerary of the $\delta-$iteration of $\tfrac{1}{q}$.
Indeed, if we denote by $\varphi_q(t)$ the absolute frequency of the
flying time $t$ when computing the $\delta-$period of $\tfrac{1}{q}$ by iterating
the $q-$Poincaré integer doubling map, the
\emph{complete numbers with respect to flying times} typically verify
\[
  \varphi_q(t) \approx 2\cdot\varphi_q(t+1)
  \qquad\text{for $t=1,2,\dots,s-1$}.
\]

To illustrate this fact consider the three \emph{complete numbers with respect to flying times}
in Table \ref{TAB::Thetable:cnwrtft}, and their flying times frequency in Table~\ref{TAB::Magick}.

\begin{table}[th]
\begin{center}
\begin{tabular}{ccc}\toprule
\multicolumn{1}{c}{\color{blue}\boldmath$q$} &
\multicolumn{1}{c}{\bfseries\color{blue}\boldmath$\ell$} &
\multicolumn{1}{c}{\bfseries\color{blue}period} \\ \midrule
4398046507391 & 42 & 2199023253695 \\
2199023246891 & 41 & 2199023246890 \\
1099511611061 & 40 & 1099511611060 \\ \bottomrule
\end{tabular}
\end{center}
\bigskip
\caption{Three \emph{complete numbers with respect to flying times} for which we compute the flying times frequency in Table~\ref{TAB::Magick} .}\label{TAB::Thetable:cnwrtft}
\end{table}

\begin{table}[th]
\begin{center}\footnotesize
\begin{tabular}{rccc}\toprule
\multirow{2}{*}{\bfseries\color{blue}t} &
\multicolumn{3}{c}{\bfseries\color{blue}flying times frequency tables} \\
& \multicolumn{1}{c}{\bfseries\color{blue}4398046507391} &
\multicolumn{1}{c}{\bfseries\color{blue}2199023246891} &
\multicolumn{1}{c}{\bfseries\color{blue}1099511611061} \\ \midrule
1  & 549754453336 & 549755811722 & 274877902765 \\
2  & 274877748292 & 274877905862 & 137438951382 \\
3  & 137438940519 & 137438952930 & 68719475692  \\
4  & 68719351661  & 68719476466  & 34359737845  \\
5  & 34359784628  & 34359738232  & 17179868923  \\
6  & 17179991043  & 17179869117  & 8589934462   \\
7  & 8589908630   & 8589934558   & 4294967230   \\
8  & 4295018499   & 4294967279   & 2147483616   \\
9  & 2147534090   & 2147483639   & 1073741807   \\
10 & 1073755996   & 1073741820   & 536870904    \\
11 & 536890826    & 536870910    & 268435452    \\
12 & 268454930    & 268435455    & 134217726    \\
13 & 134216871    & 134217728    & 67108863     \\
14 & 67109504     & 67108863     & 33554432     \\
15 & 33556351     & 33554432     & 16777215     \\
16 & 16774767     & 16777216     & 8388608      \\
17 & 8390697      & 8388608      & 4194304      \\
18 & 4193034      & 4194304      & 2097152      \\
19 & 2096089      & 2097152      & 1048576      \\
20 & 1048587      & 1048576      & 524288       \\
21 & 524142       & 524288       & 262144       \\
22 & 262084       & 262144       & 131072       \\
23 & 131180       & 131072       & 65536        \\
24 & 65565        & 65536        & 32768        \\
25 & 32649        & 32768        & 16384        \\
26 & 16458        & 16384        & 8192         \\
27 & 8190         & 8192         & 4096         \\
28 & 4066         & 4096         & 2048         \\
29 & 2055         & 2048         & 1024         \\
30 & 1005         & 1024         & 512          \\
31 & 506          & 512          & 256          \\
32 & 247          & 256          & 128          \\
33 & 129          & 128          & 64           \\
34 & 60           & 64           & 32           \\
35 & 36           & 32           & 16           \\
36 & 16           & 16           & 8            \\
37 & 9            & 8            & 4            \\
38 & 6            & 4            & 2            \\
39 & 3            & 2            & 1            \\
40 & 2            & 1            & 1            \\
41 & 1            & 1            &              \\
42 & 1            &              &              \\ \bottomrule
\end{tabular}
\end{center}
\bigskip
\caption{The flying times frequency tables of the three \emph{complete numbers with respect to flying times} in Table \ref{TAB::Thetable:cnwrtft}.
Observe that the relation $\varphi_q(t) \approx 2\cdot\varphi_q(t+1)$ is verified with {\it miraculous} precision,
specially for high values of $\varphi_q(t)$.}\label{TAB::Magick}
\end{table}

To determine the value of $\kappa$ we have run (a variant of) the program

\centerline{\texttt{\color{blue}compute\_all\_periods\_in\_a\_range\_of\_odd\_qs.c}}

\noindent{}(that implements Algorithm~\ref{ALGO::ElBo}),
with $\kappa \in \{1,2,\dots,12\}$,
on the three \emph{complete numbers with respect to flying times} discussed above.
In Figure~\ref{FIG::kappacomputimes} below, we show the obtained computation times
for the different values of $\kappa$.
Notice that the plot of the computation times versus
$\kappa$ shows the expected ``U'' shaped graph, and
indicates that the best performing $\kappa$ value is $\kappa = 2.$

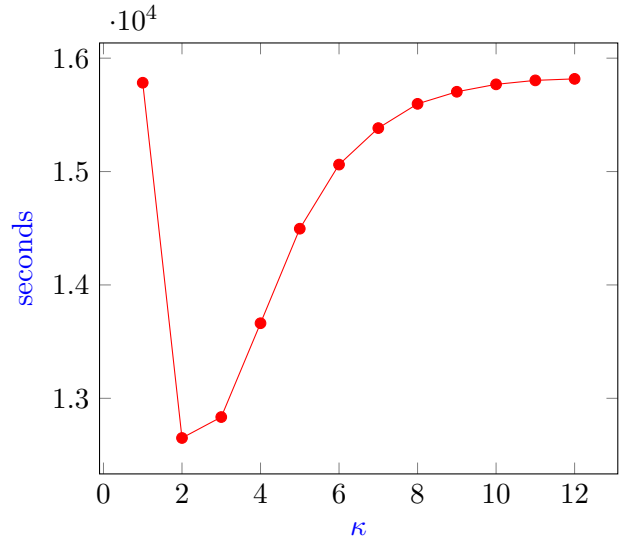
\begin{figure}[ht]
\begin{center}
\begin{tikzpicture}
    \begin{axis}[xlabel={\textcolor{blue}{$\kappa$}}, ylabel={\textcolor{blue}{seconds}}]
        \addplot[red, mark color=red, mark=*] coordinates { (1, 15783.3) (2, 12650) (3, 12834.7) (4, 13662.2)
                               (5, 14495.9) (6, 15061.4) (7, 15383.2) (8, 15597.2)
                               (9, 15703.7) (10, 15769.1) (11, 15804.1) (12, 15817.4)};
     \end{axis}
\end{tikzpicture}
\end{center}
\caption{Computation times, for different values of $\kappa$,
         of the periods of three \emph{complete numbers with respect to flying times}.
         Observe the expected ``U shaped'' graph and that the best performance $\kappa$
         value is two.
         These computations have been done in parallel in a server
         Dell PowerEdge R6625 with 32 AMD EPYC 9124 16-Core Processor at 3000MHz and 2 threads per core.
         The RAM memory is DDR5 at 4800 MT/s, and 12 memory channels per processor.}\label{FIG::kappacomputimes}
\end{figure}

To test Algorithm~\ref{ALGO::ElBo} (and the $\kappa=2$ determination) we have computed
a number of periods of odd $q$'s of different magnitudes.
In Table~\ref{TAB::Thetable} one can find a representative list of periods,
together with the computation times in seconds.
Additionally, in

\centerline{\color{blue}\url{https://github.com/CombTopDynamics-cat/Mersenne-doubling.git}}

\noindent one can find the following four categorised files of computed periods
(in '\texttt{.gz}' or '\texttt{.bz2}' format and with self descriptive names):
\begin{itemize}
\item \texttt{\color{blue}Large\_PrimePeriods\_segmented\_and\_sorted.lis}:
      contains 12,135 odd values of $q$ with the corresponding periods, that turned out to be prime numbers.
      Every one of the pairs $(q, \text{period})$ in this file provides an example of a
      Mersenne number $M(\text{period})$ with prime exponent which is composite (because it has a divisor $q$).
      The file is sorted with respect to periods.
      Large means that the periods in the file are larger than 136,279,841 which is the exponent
      of the last known Mersenne Prime in the \emph{List of Known Mersenne Prime Numbers} \cite{GIMPS}.
\item \texttt{\color{blue}RelativelySmallPrimePeriods\_segmented\_and\_sorted.lis}:
      contains 121,104 odd values of $q$ with the corresponding periods, that turned out to be prime numbers.
      As the previous file every pair $(q, \text{period})$ in this file provides an example of a
      composite Mersenne number with prime exponent.
      The file, again, is sorted with respect to periods.
\item \texttt{\color{blue}OddPeriods\_NotPrime\_segmented\_and\_sorted.lis}:
      contains 2,199,404 odd values of $q$ with the corresponding periods, that turned out to be odd but not prime.
      Again, the file is sorted with respect to periods.
\item \texttt{\color{blue}EvenPeriodsDataBase\_segmented\_and\_sorted\_with\_respect\_to\_Qs.dat}: it
      contains 32,734,191 odd values of $q$ with the corresponding periods, that turned out to be even.
      This file, of lesser interest, is sorted with respect to the values of $q$.
\end{itemize}
The word \emph{segmented} in the filenames refers to the column \texttt{\color{blue}\#segm}.
When a register has the first field (\texttt{\color{blue}\#segm}) equals to $s$,
the number $q$ in the register verifies $2^{s-1} < q < 2^{s}.$

\begin{table}[th]
\begin{center}
\begin{tabular}{lrl}\toprule
\multicolumn{1}{c}{\color{blue}\boldmath$q$} &
\multicolumn{1}{c}{\bfseries\color{blue}period} &
\multicolumn{1}{c}{\bfseries\color{blue}computation time} \\ \midrule
4398046508903 & 2199023254451 & \parbox{6.5em}{\raggedleft$4979.3$} \\
4398046507367 & 2199023253683 & \parbox{6.5em}{\raggedleft$4994.4$} \\
2199023252039 & 1099511626019 & \parbox{6.5em}{\raggedleft$2512.6$} \\
2199023249903 & 1099511624951 & \parbox{6.5em}{\raggedleft$2493.28$} \\
2199023247479 & 1099511623739 & \parbox{6.5em}{\raggedleft$2511.33$} \\
2199023246807 & 1099511623403 & \parbox{6.5em}{\raggedleft$2497.67$} \\
1099511615303 &  549755807651 & \parbox{6.5em}{\raggedleft$1246.68$} \\
1099511614679 &  549755807339 & \parbox{6.5em}{\raggedleft$1256.04$} \\
1099511611607 &  549755805803 & \parbox{6.5em}{\raggedleft$1247.29$} \\
1099511610839 &  549755805419 & \parbox{6.5em}{\raggedleft$1250.74$} \\ \midrule
   4291783591 &     143059453 & \parbox{6.5em}{\raggedleft$0.354289$} \\
   4291592791 &     143053093 & \parbox{6.5em}{\raggedleft$0.342082$} \\
   4291434391 &     143047813 & \parbox{6.5em}{\raggedleft$0.324896$} \\ \midrule
4398046511103 &            42 & \parbox{6.5em}{\raggedleft$0$}        \\
4398046511101 &    1938935328 & \parbox{6.5em}{\raggedleft$4.44301$}  \\
4398046511099 &  146553656916 & \parbox{6.5em}{\raggedleft$334.145$}                 \\
4398046511097 &    5511336480 & \parbox{6.5em}{\raggedleft$12.4794$}  \\ \bottomrule
\end{tabular}
\end{center}
\bigskip
\caption{A representative list of periods computed with Algorithm~\ref{ALGO::ElBo}
(of course with the right determination of $\kappa$), together with the computation times in seconds.
These computations have been done in a server
Dell PowerEdge R6625 with 32 AMD EPYC 9124 16-Core Processor at 3000MHz and 2 threads per core.
The RAM memory is DDR5 at 4800 MT/s, and 12 memory channels per processor.
Observe that the computation of periods for values $q > 2^{40}$ starts to be demanding in general.}\label{TAB::Thetable}
\end{table}


\subsection{Detecting composite Mersenne numbers by finding one of its divisors}

This application of Theorem~\ref{MainTheorem::A} and the algorithms described in the previous section
consists in studying the (non) primality of $M(n)$ without the explicit computation of $M(n)$.
This is significantly useful for large values of $n$ since,
the proposed methodology allows us to \emph{easily} discard
a large number of potential prime Mersenne numbers
(i.e. composite Mersenne numbers with $n$ prime)
for large values of $n$.
Here, \emph{large} means larger than the exponent
of the 52th Mersenne Prime, $M(136,279,841)$ \cite{GIMPS},
which has about 41 million decimals digits.

The ultimate goal is, of course, to discover new Mersenne primes.
However, having an efficient method to rule out candidates can be equally valuable.
Indeed, an effective algorithm for identifying composite Mersenne numbers
reduces the pool of candidates that need to be tested for primality.

Our \emph{Mersenne non-primality detector} works as follows:
\begin{enumerate}[\bfseries{}Step 1:]\setcounter{enumi}{-1}
\item Choose an odd number $q$ such that
      $q \geq \max\bigl\{5, 2^{\kappa}\bigr\} = 5$
      (recall that we aim at using big magnitude numbers), and
      $q \leq 2^{64}-1$ (because in the current application we are using unsigned integers of 64 bits).
\item Compute the $\delta-$period $n$ of $\frac{1}{q}\in \SI$ (Algorithm~\ref{ALGO::ElBo}).
\item Check if $n$ is a prime number (otherwise, $M(n)$ is, trivially, composite).
\item In the affirmative, Theorem~\ref{MainTheorem::A}(a) tells us that $q$ is a divisor of $M(n)$.
      Thus, $M(n)$ is composite with $n$ being prime.
\end{enumerate}

\begin{remark}
The limitation $q \leq 2^{64}-1$ could, in principle, be overcome by using arbitrary-precision arithmetic.
However, as observed in Table~\ref{TAB::Thetable}, the computation of periods for values $q > 2^{40}$
starts to be really demanding.
In other words the data-type bound $q \leq 2^{64}-1$ can certainly be overcome but not at a cheap price.
\end{remark}

In Table \ref{table:eleven} we list the first eleven lines of the file

\centerline{\texttt{\color{blue}Large\_PrimePeriods\_segmented\_and\_sorted.lis}}

\noindent{}in the public repository

\centerline{\color{blue}\url{https://github.com/CombTopDynamics-cat/Mersenne-doubling.git},}

\noindent{}with the approximate number of digits of $M(\mbox{period})$ added in the third column.
Recall that this file, obtained with a massive exploration using Algorithm~\ref{ALGO::ElBo},
in fact contains
12,135 odd values of $q$ with the corresponding periods, that turned out to be prime numbers.
Consequently, \emph{we have obtained 12,135 Mersenne non-prime numbers,
all of them with prime exponent larger than $136,279,841$}.

One of the most striking examples of the potential of the
\emph{Mersenne non-primality detector} is the following:
The number $q=4,398,046,508,903$ divides $M(2,199,023,254,451)$,
where $2,199,023,254,451$ is a prime number.
Notably, $M(2,199,023,254,451)$ has approximately $6.6\times 10^{11}$ decimal digits,
yet we were able to obtain a non-trivial divisor without ever computing its explicit value.
Moreover, as explained in Table~\ref{TAB::Thetable},
the CPU time required to determine that the $\delta-$period ($n=2,199,023,254,451$) of $\frac{1}{q}\in \SI$ ($q=4,398,046,508,903$)
is $4979.3$ seconds.

\begin{table}[th]
\begin{center}
\begin{tabular}{lrl}\toprule
\multicolumn{1}{c}{\color{blue}\boldmath$q$} &
\multicolumn{1}{c}{\bfseries\color{blue}period} &
\multicolumn{1}{c}{\bfseries\color{blue}\#Digits of \boldmath$M(\mbox{period})$} \\ \midrule
4398046508903 & 2199023254451 & \parbox{9em}{\raggedleft$6.6\times 10^{11}$} \\
4398046507367 & 2199023253683 & \parbox{9em}{\raggedleft$6.6\times 10^{11}$} \\
2199023252039 & 1099511626019 & \parbox{9em}{\raggedleft$3.3\times 10^{11}$} \\
2199023249903 & 1099511624951 & \parbox{9em}{\raggedleft$3.3\times 10^{11}$} \\
2199023247479 & 1099511623739 & \parbox{9em}{\raggedleft$3.3\times 10^{11}$} \\
2199023246807 & 1099511623403 & \parbox{9em}{\raggedleft$3.3\times 10^{11}$} \\
1099511615303 &  549755807651 & \parbox{9em}{\raggedleft$1.65\times 10^{11}$} \\
1099511614679 &  549755807339 & \parbox{9em}{\raggedleft$1.65\times 10^{11}$} \\
1099511611607 &  549755805803 & \parbox{9em}{\raggedleft$1.65\times 10^{11}$} \\
1099511610839 &  549755805419 & \parbox{9em}{\raggedleft$1.65\times 10^{11}$} \\
2199023253737 &  274877906717 & \parbox{9em}{\raggedleft$8.27\times 10^{10}$} \\  \bottomrule
\end{tabular}
\end{center}
\bigskip
\caption{The second column lists the eleven largest (among the 12,135) prime $\delta$-periods
found in our large-scale, though non-exhaustive, exploration using Algorithm~\ref{ALGO::ElBo}.
By Theorem A(a), the corresponding $M(\mathrm{period})$ values are composite Mersenne numbers
whose exponents are prime and significantly larger than $136{,}279{,}841$.
For comparison, the 52nd known Mersenne prime contains about $4.1\times 10^6$ digits,
whereas the numbers arising from our computations, shown in the third column,
reach sizes of up to $6.6\times 10^{11}$ digits.} \label{table:eleven}
\end{table}

Step~2 of the \emph{Mersenne non-primality detector} above deserves some attention (and a couple of comments).
First, notice that we have completely separated the efficient calculation
of a $\delta-$period from the determination of its primality.
In other words, for efficiency, we have postponed the determination
of the primality of the calculated periods to the
post-processing phase.

On the other hand,
the \emph{Mersenne non-primality detector} has been formulated
in terms of the individual calculation
of $\delta-$periods but, as already explained,
is in fact based on a massive quest
for prime $\delta-$periods.
In other words,
in the post-processing phase we need to do
a massive check of the primality of the calculated periods,
thus benefiting from the decision already made
to postpone the determination of the primality of the $\delta-$periods to the
post-processing phase.

It turns out that a massive determination of the primality of periods
smaller than or equals to
$$
965,211,250,482,432,409 = 982,451,653^2 > 2^{59}
$$
is very easy and quick, by using the naivest possible algorithm
with the help of the public table containing the first
fifty million primes, whose last member is $982,451,653$
(taken from {\color{blue}\emph{The PrimePages: prime number research \& records}
\url{https://t5k.org/}}).
This is implemented in the program

\centerline{\texttt{\color{blue}IsPrime\_geek.c}}

\noindent{}posted (as the rest of programs and data from this paper) at the public repository

\centerline{\color{blue}\url{https://github.com/CombTopDynamics-cat/Mersenne-doubling.git}}

\noindent{}As already explained, the program \texttt{\color{blue}IsPrime\_geek.c}
determines the primality of a (big) set of positive integers
with the help of a built-in (at compile time) table with all 49,999,999
(sorted) prime numbers from 3 to $982,451,653$.
The algorithm to determine the primality of a candidate $n$ between 3 and
$965,211,250,482,432,409$ is divided into three cases:
\begin{enumerate}[\bfseries{Case} 1:]
\item If $n$ is even it is not prime.
\item If $n$ is odd, $n \leq 982,451,653$, then the primality of $n$ is determined
      by a binary search of $n$ in the built-in list of
      49,999,999 consecutive sorted primes from 3 to $982,451,653$.
\item If $n$ is odd and  $982,451,653 < n \leq 965,211,250,482,432,409 = 982,451,653^2$,
      then $\sqrt{n} \leq 982,451,653.$ Hence, $n$ is composite if and only if
      there exists a prime number in the table, smaller than or equals to $\sqrt{n}$,
      dividing $n$.
\end{enumerate}

\subsection{Conclusions}
We end this section with several remarks and consequences drawn from the preceding discussion, where we developed the algorithms used to detect large — indeed, extremely large — composite Mersenne numbers and presented some notable computational results.

On the one hand, the smallest prime $\delta-$period in the file

\centerline{\texttt{\color{blue}Large\_PrimePeriods\_segmented\_and\_sorted.lis}}

\noindent is $143,047,813$ (that corresponds to $q = 4,291,434,391$).
An artifact of our search procedure is that, at this moment,
we do not know which prime numbers (strictly) between $136,279,841$ and $143,047,813$ happen to be $\delta-$periods of $\frac{1}{q}\in\SI$ for some odd $q$.
This would give a whole range of composite Mersenne numbers with prime exponents.
Furthermore, we also do not know (and probably is false) whether the periods from
that file are consecutive in the set of prime numbers.
In the affirmative we would have again a range of composite Mersenne numbers with prime exponents. These ideas give a possible way to greatly improve the conclusions of our search:
focus on ranges of consecutive prime periods rather on individual periods themselves.

On the other hand, Table~\ref{TAB::Thetable} contains several messages.
One of them is that the periods of the doubling map and their computation times,
in general do not depend on the magnitude of the odd number $q$.
To be more specific observe that
$q = 4,398,046,511,103 = 2^{42}-1$ has period $42$
(that is, $\frac{1}{q}$ has period $42$ under the doubling map),
but $q = 4,398,046,511,101 = 2^{42}-3$ has period $1,938,935,328$.
This illustrates the extreme dependence to initial conditions of the orbits of the doubling map
--- a very well known phenomena (see \cite{Bob_Book}).
However, the computation times are very small in both cases,
despite the fact that the second period is about $46,165,126$ times larger than the first.
This observation has a twofold interpretation.
On the one hand, although the second period is enormous,
its associated dynamical itinerary must possess a certain
simple internal structure.
Equivalently, the number of iterates of the
$q-$Poincaré integer doubling map must be small \emph{in some sense}.
On the other hand, the algorithm used to compute the periods
is extremely efficient for certain ``simple'' periods (itineraries).

\section{On the primality of Mersenne numbers with prime exponent}\label{SEC::Prime_Mersenne_numbers}

In Section~\ref{SEC::Algos}, we presented the algorithms developed to compute the period of a given element $\frac{1}{q} \in \SI$, where $q$ is an odd integer. We then applied these methods to the detection of composite Mersenne numbers with large prime exponents, many of them far exceeding the exponent of the 52nd known Mersenne prime.

In the present section, we turn to the central problem about Mersenne numbers: the development of a primality test capable of determining whether a Mersenne number with prime exponent is itself prime. Within the framework of our approach — namely, through the algorithms introduced in Section~\ref{SEC::Algos} — this problem becomes substantially more difficult. Indeed, for large values of $n$, the method requires the computation of $\delta$-periods associated with values of $q$ that are significantly larger than those considered previously. Consequently, both the computational complexity and the arithmetic intricacy of the problem increase dramatically.

\subsection{Preliminary results}

Before presenting the primality test and concluding with few remarks we state and prove (for shake of completeness) results relating the period doubling map $\delta$ and some well-known number theoretical propositions.

\begin{lemma}\label{lemma:period}
Let $q \in \N$ be an odd number and let $p \in \{1,2,\dots,q-1\}$ be such that $\gcd(p,q) = 1$.
Then, $\tfrac{p}{q}\in \SI$ is $\delta-$periodic, and
\[
\operatorname{Orb}_{\delta}\bigl(\tfrac{p}{q}\bigr) = \elasticset{\frac{r_k}{q}}{r_k \in \{1,2,\dots,q-1\}\text{ and }\gcd(r_k,q) = 1}.
\]
\end{lemma}

\begin{proof}
Since $q$ is odd, we know from Lemma~\ref{LEM::CarlesonGamelin} that $\tfrac{p}{q}\in \SI$ is $\delta-$periodic.
For every integer $k\geq 0$ we can write
\begin{equation}\label{eq:div}
      2^k p = q \alpha_k + r_k
\end{equation}
for some $\alpha_k \in \Z^+$ and $r_k \in \{1,2,\ldots, q-1\}$.
Thus,
\[
\operatorname{Orb}_{\delta}\bigl(\tfrac{p}{q}\bigr) =
    \elasticset{2^k \tfrac{p}{q} \pmod{1}}{k \in \Z^+} =
    \elasticset{\frac{r_k}{q}}{r_k \in \{1,2,\dots,q-1\}}.
\]

By way of contradiction assume that $\gcd(r_k, q) = \ell > 1$ for some
$r_k$ such that $\frac{r_k}{q} \in \operatorname{Orb}_{\delta}\bigl(\tfrac{p}{q}\bigr).$
This implies that $\ell$ divides $q$ and $r_k$. Moreover, since $q$ is odd, $\ell$ is also odd.
From \eqref{eq:div} we conclude that $\ell$ divides $p$; a contradiction with $\gcd(p, q) = 1$.
\end{proof}

Lemma above can be used to prove the following well-known number theoretical statement, using the $\delta$-map.

\begin{lemma} \label{lem:number_doubling}
Let $n>2$ be a prime number.
\begin{enumerate}
\item[(a)]  $2^{n-1} = 1 \, \pmod{n}$. (Fermat's little Theorem)
\item[(b)] If $q>2$ is a prime divisor of $M(n)$ then $q=1+2n\ell$ for some $\ell\in \N$.
\item[(c)] If $q>2$ is a prime divisor of $M(n)$ then $q = \pm 1  \, \pmod{8}$.
\end{enumerate}
\end{lemma}

\begin{proof}
We prove statement (a). From \eqref{eq:per_points} we know that the number of periodic points of period $n$ of the doubling map is equal to $2^{n}-2$ . Since $n$ is a prime number, all these $2^n-2$ points  have minimal period $n$, so they form $M\geq 1$ orbits of minimal period $n$. Obtaining that $2^{n}-2=n M$. Using again that $n$ is a prime number we conclude that $M$ is even and so $M/2$ is an integer number. Finally, we have that  $2^{n-1} = \frac{M}{2} n +1$, and the result follows.

We prove statement (b). Since $q>2$ is a prime divisor of $M(n)$ we know from
Lemma~\ref{lemma:period} that $\{j/q,\ j=1,\ldots q-1\}$ are $q-1$ periodic points of minimal period $n$ (notice that $\gcd(j,q)=1$). These $q-1$ periodic points form $L\geq 1$ periodic orbits of minimal period $n$, obtaining that $q-1 = Ln$. Since $q$ and $n$ are prime numbers, we have that $2$ divides $q-1$ and therefore $2$ divides $L$. Thus, $L=2\ell$ and we have that $q = 1 + 2 n   \ell$.

We prove statement(c). By assumption $q>2$ is a divisor of $M(n)=2^n-1$, or equivalently, $2^n = 1 \pmod{q}$. Multiplying by two the above equality we obtain that
$ 2^{n+1}=2 \, \pmod{q}$. Using that $n$ is a (odd) prime number we  conclude that $2^{(n+1)/2}$ solves de congruence $x^2 \equiv 2 \, \pmod{q}$.  Finally, by quadratic reciprocity given a prime number $q$ the congruence $x^2=2 \pmod{q}$ is solvable if and only if $n$ is congruent to $\pm 1 \pmod{8}$, and the result follows.
\end{proof}

\subsection{A primality test}

Fix $n_0\geq 3$ prime.
Our goal is to define a test which identifies all Mersenne primes
$M(j),\ j \in \{3,\ldots,  n_0\}$.

We first introduce some notation. We denote by $\N_{\rm pMd}$ (prime Mersenne divisors)
the set of natural numbers greater or equal to 3 congruent to $\pm 1  \, \pmod{8}$, and we define the set $\mathcal O(n_0)$ to be
\[
\mathcal O(n_0) := \{  q \in \N_{\rm pMd} \ | \   q \leq \lfloor \sqrt{M(n_0)} \rfloor \}
\]
For every $q \in \mathcal O(n_0)$  we denote by $n(q)$ the $\delta-$period of $\frac{1}{q}\in \SI$. We also define
\[
A(j):=\{q \in \mathcal O(n_0) \, | \,   n(q)=j \}, \ j=3,\ldots, n_0,
\]
and denote by $\lvert A(j) \rvert$ its cardinality. Next lemma characterises the prime Mersenne numbers below any given $n_0\geq 3$. So, it defines a primality test.

\begin{lemma}\label{lem:primality_test}
Let $n_0\geq 3$ be a prime number. Set
\[
v(j) := \lvert A(j) \rvert.
\]
Then, for every $j\in\{3,\ldots,n_0\}$ prime we have that  $M(j)$ is prime if and only if $v(j)=0$, or $v(j)=1$ with $M(j)\leq \sqrt{M(n_0)}$.
\end{lemma}

\begin{proof}
Fix $j\in\{3,\ldots, n_0\}$. From Lemma~ \ref{lem:number_doubling}, it follows that the only possible odd divisors of Mersenne numbers
$M(n)$ with $n\leq n_0$ are those $q\in \mathcal O(n_0)$.

We first assume that $M(j)$ is a prime number. If $M(j) > \sqrt{M(n_0)}$ Theorem~\ref{MainTheorem::A}(b) implies that for every $q \in \mathcal O(n_0)$ the period $n(q)$  of $\theta=\frac{1}{q}$  under the doubling map verifies that  $n(q) \neq j$ and thus $v(j)=0$. If $M(j) \leq \sqrt{M(n_0)}$ then $M(j)\in \mathcal O(n_0)$ and $\theta=\frac{1}{q}$ with $q=M(j)$ has period $j$ under the doubling map. Thus $v(j)=1$ (and $M(j) \leq \sqrt{M(n_0)}$).

To prove the converse implication, we first assume that $v(j)=0$. In this case, for all $q\in\mathcal O(n_0)$ we have that $\theta=\frac{1}{q}$ has period $n(q)\ne j$. Therefore, by Theorem~\ref{MainTheorem::A}(b), it follows that $M(j)$ is prime. Now we suppose that $v(j)=1$ and that $M(j) \leq \sqrt{M(n_0)}$. In this situation, the only $q\in \mathcal O(n_0)$ for which $\theta=\frac{1}{q}$ has period $n(q)=j$ is $q=M(j)$. Hence, the only divisor of $M(j)$ different from 1 is $M(j)$ itself, which implies that $M(j)$ is prime.
\end{proof}

\begin{remark}
Observe that the case $v(j)\geq 2$ immediately implies that $M(j)$ is composite, since it guarantees the existence of at least one divisor $q\ne M(j)$.
\end{remark}

As an example we show the primality test for $n_0=31$. We have $\lfloor \sqrt{M(31)} \rfloor =46340$ and so we need to consider $ q$ in $\mathcal O(n_0) = \{  \pm 1 + 8 \ell \, , \,   \ell =1, \ldots, 5793\}$.  The output of the test gives
\begin{itemize}
\item[(a)] $v(3)=v(5)=v(7)=v(13)=1$ with $M(j)<\lfloor \sqrt{M(31)} \rfloor,\ j=3,5,7,13$.
\item[(b)] $v(11)=3$.
\item[(c)] $v(17)=v(19)=0$ (notice that $M(17)>\lfloor \sqrt{M(31)} \rfloor$).
\item[(d)] $v(23)=1$ (with $M(23)>\lfloor \sqrt{M(31)} \rfloor$).
\item[(e)] $v(29)=3$.
\item[(f)] $v(31)=0$.
\end{itemize}

Thus, applying Lemma~\ref{lem:primality_test}, we obtain that $M(j),\ j=3,5,7,13,17,19,31$ are prime while $M(j),\ j=11,23,29$ are composite.

\subsection{Conclusions}

As we explained in Section \ref{SEC::Algos}, using the unsigned long integers type in our algorithms bounds the potential computations we can do. In this case we can only implement the test up to $n_0=128$. This is so because the algorithm requires to check $q\in\mathcal O(n_0)$, or equivalently, $q\leq \lfloor \sqrt{M(n_0)} \rfloor$, and the long integers type requires $\lfloor \sqrt{M(n_0)} \rfloor \lesssim 2^{64}$.

\bibliographystyle{alpha}
\bibliography{biblio}

\end{document}